\documentclass[11pt]{amsart}

\usepackage{amsmath}
\usepackage{amsfonts}
\usepackage{amssymb}
\usepackage{amsthm}

\usepackage{graphicx}

\usepackage[T1]{fontenc}

\usepackage{enumerate}

\usepackage{color}

\def\CC{\mathbb{C}} 
\def\DD{\mathbb{D}} 
\def\NN{\mathbb{N}} 
\def\ZZ{\mathbb{Z}} 
\def\TT{\mathbb{T}} 
\def\CCS{\widehat{\CC}} 
\def\PD{\mathbb{P}} 
\def\CPD{\overline{\mathbb{P}}} 
\def\CDD{\overline{\DD}} 
\def\OO{\mathcal{O}} 
\def\CL{\mathcal{C}} 
\def\INT{\textnormal{int}\,} 

\def\CCO{\textnormal{cc}_0\,} 
\def\BNU{\overline{\nu}} %
\def\NONES#1{\textnormal{\textbf{1}}_{#1}}

\def\HULL#1#2{\widehat{#1}_{#2}} 
\def\HULLP#1{\HULL{#1}{}} 

\newtheoremstyle{remarkstyle}{}{}{}{}{\bf}{.}{ }{}

\newtheorem{THE}{Theorem}[section]

\newtheorem{LEM}[THE]{Lemma}

\newtheorem{OBS}[THE]{Observation}

\theoremstyle{remarkstyle}
\newtheorem{RM}[THE]{Remark}

\newtheorem*{DEF}{Definition}
\newtheorem*{AS}{Assumption}

\begin{document}


\title[The $*$-product of domains in $\CC^2$]{The $*$-product of domains in several complex variables}

\author{Sylwester Zaj\k{a}c}
\address{Institute of Mathematics, Faculty of Mathematics and Computer Science, Jagiellonian University, \L ojasiewicza 6, 30-348 Krak\'ow, Poland}
\email{sylwester.a.zajac@gmail.com}

\thanks{The research was supported by NCN grant SONATA BIS no. 2017/26/E/ST1/00723 of the National Science Centre, Poland.}

\keywords{Hadamard product, analytic continuation, spaces of holomorphic functions}
\subjclass[2010]{Primary: 32A05, 32D15}

\begin{abstract}
In this article we continue the research, carried out in \cite{zajac}, on computing the $*$-product of domains in $\CC^N$.
Assuming that $0\in G\subset\CC^N$ is an arbitrary Runge domain and $0\in D\subset\CC^N$ is a bounded, smooth and linearly convex domain (or a non-decreasing union of such ones), we establish a geometric relation between $D*G$ and another domain in $\CC^N$ which is 'extremal' (in an appropriate sense) with respect to a special coefficient multiplier dependent only on the dimension $N$.
Next, for $N=2$, we derive a characterization of the latter domain expressed in terms of planar geometry.
These two results, when combined together, give a formula which allows to calculate $D*G$ for two-dimensional domains $D$ and $G$ satisfying the outlined assumptions.
\end{abstract}

\maketitle

\section{Introduction}

Let $\OO_0$ be the set of all germs of holomorphic functions at the origin of $\CC^N$ and let $\OO_{0,D}$, for a domain $0\in D\subset\CC^N$, be the subset of $\OO_0$ consisting of all germs of elements of $\OO(D)$.
The latter symbol denotes, as usually, the Fr\'{e}chet space of all holomorphic functions on $D$ equipped with the compact-open topology.
The Hadamard product, called also the $*$-product, can be regarded as a bilinear mapping from $\OO_0\times\OO_0$ to $\OO_0$ given by the formula
$$\left(\sum_{\alpha\in\NN^N}f_\alpha z^\alpha\right)*\left(\sum_{\alpha\in\NN^N}g_\alpha z^\alpha\right):=\sum_{\alpha\in\NN^N}f_\alpha g_\alpha z^\alpha.$$
It was extensively studied in various aspects: as a bilinear form, as a linear operator with one factor fixed (see the survey \cite{render97} and, for instance, the papers \cite{bergoncalmorprabas}, \cite{bruckmuller92}, \cite{bruckmuller95}, \cite{bruckrender00}, \cite{grosseerdmann}, \cite{muller92}, \cite{muller93}, \cite{mullernaikponnusany06}, \cite{mullerpohlen10}, \cite{mullerpohlen12}, \cite{render96}, \cite{rendersauer96a}, \cite{rendersauer96b}, \cite{zajac}), and also as a map acting on spaces of real analytic functions (see \cite{domanskilangenbruch12a}, \cite{domanskilangenbruch12b}, \cite{domanskilangenbruch13}, \cite{domanskilangenbruch_pre1}, \cite{domanskilangenbruch_pre2}, \cite{domanskilangenbruchvogt}).
The problem of similar nature as here, that is of extending the $*$-products to as large domain as possible, was investigated, for instance, in \cite{aizenbergleinartas83}, \cite{elin94} and \cite{leinartas89} for weighted Hadamard products with certain weights, where results were obtained for starlike domains and for so-called $p$-convex domains (we refer the reader to \cite{elin94}).

In \cite{zajac} it was shown (see Proposition 4.1 there) that if at least one of domains $0\in D, G\subset\CC^N$ is Runge domain, then there exists the largest domain $0\in\Omega\subset\CC^N$ having the property that the image of $\OO_{0,D}\times\OO_{0,G}$ under the $*$-product lies in $\OO_{0,\Omega}$ (here, as well as in \cite{zajac}, we follow \cite[Definition 2.7.1]{hormander_introduction} and consider Runge domains to be pseudoconvex).
This largest $\Omega$, denoted as $D*G$, was the subject of research carried out in \cite{zajac}, which concluded in finding a description of $D*G$ for domains of a special class.
The methodology employed in \cite{zajac}, as well as the results achieved there, were of symmetrical nature: both $D$ and $G$ were assumed to belong to the same family of domains and the fundamental integral formula for $f*g$ relied on geometric properties of $D$ to the same extent as on those of $G$.

The approach presented in this paper is substantially different.
Starting from a non-symmetric integral expressing $f*g$, we investigate $D*G$ with only $D$ being of the same particular class as in \cite{zajac} and $G$ being an arbitrary Runge domain.
These considerations conclude in Theorem \ref{th_D_s_G_when_D_from_DN}, which establishes a relation, expressed in geometric terms, between $D*G$ and $h_{\NONES{N}}*G$.
Here $h_{\NONES{N}}(z)=(1-z_1-\ldots-z_N)^{-N}$ and the set $h_{\NONES{N}}*G$ is the largest domain containing the origin on which every product $h_{\NONES{N}}*g$, for $g\in\OO_{0,G}$, can be analytically continued.
When $G$ is Runge domain, existence of $h_{\NONES{N}}*G$ is guaranteed by the aforementioned \cite[Proposition 4.1]{zajac}.
To calculate $D*G$ we must, however, face the problem of computing $h_{\NONES{N}}*G$.
We deal with this topic in Section \ref{sect_h11_s_D_for_Runge}, where, in Theorem \ref{th_description_of_h11_s_D}, we derive a nice geometric formula for $h_{\NONES{N}}*G$, but, unfortunately, only for $N=2$.
It is worthy to emphasize that although obtaining a candidate for this set was quite straightforward and relied mainly of certain integral formula, main difficulties were met in demonstrating that this candidate is the largest one on which all $h_{\NONES{2}}*g$'s extend.
Combining this result with Theorem \ref{th_D_s_G_when_D_from_DN} announced above leads to a complete description of $D*G$ for two-dimensional domains satisfying the listed assumptions.
Nevertheless, the question for higher dimensions remains open.

\section{Preliminaries}\label{sect_preliminaries}

We begin by introducing basic concepts and notation setting grounds for this study.
We use the standard symbols $\DD$, $\TT$, $\CC_*$ and $\CCS$ to denote, respectively, the unit disc in $\CC$, its boundary, the punctured plane $\CC\setminus\lbrace 0\rbrace$ and the Riemann sphere $\CC\cup\lbrace\infty\rbrace$.
We assume that the set $\NN$ of all natural numbers contains $0$.
By $\PD_N(z,r)$ and $\CPD_N(z,r)$ we mean the open and closed polydiscs centered at $z$ and having the radius $r$.

To shorten notation, we will use the word 'loop' to denote a continuous map defined on $\TT$ and the word 'smooth' to declare being of the $\CL^\infty$ class.

For points $z=(z_1,\ldots,z_N),w=(w_1,\ldots,w_N)\in\CC^N$, by $z\bullet w$ we denote the product $z_1 w_1+\ldots+z_N w_N$ and by $z\cdot w$ or $zw$ we denote their coordinate-wise product, that is, the point $(z_1 w_1,\ldots, z_N w_N)$.
The identity element of the latter multiplication, $(1,1,\ldots,1)$ ($N$ times), is called $\NONES{N}$.
Given two sets $A,B\subset\CC^N$, by $AB$ or $A\cdot B$ we mean their algebraic product, i.e. the set $\lbrace ab:a\in A,b\in B\rbrace$.
We use the classical notation of exponentiation, where for an $\alpha=(\alpha_1,\ldots,\alpha_N)\in\ZZ^N$ the symbol $z^\alpha$ denotes $z_1^{\alpha_1}\ldots z_N^{\alpha_N}$, holding the convention that $z_j^0=1$.
Moreover, $\alpha!$ and $|\alpha|$ will, as usual, denote the product $\alpha_1!\ldots\alpha_N!$ and the sum $\alpha_1+\ldots+\alpha_N$.

For compact sets $K\subset\CC^N$ and $L\subset\Omega$ we use the standard notation $\HULLP{K}$ and $\HULL{L}{\Omega}$ for the polynomial hull of $K$ and the holomorphic hull of $L$ with respect to a domain $\Omega$.
The supremum of modulus of a complex-valued function over a set $A\subset\CC^N$ is denoted by $\|f\|_A$.
Finally, if $0\in A$, then by $\CCO A$ we understand the connected component of $A$ containing $0$.

\subsection{Integral formula for the $*$ product}

Take two power series $$f(z)=\sum_{\alpha\in\NN^N}f_\alpha z^\alpha,\quad g(z)=\sum_{\alpha\in\NN^N}g_\alpha z^\alpha$$ convergent in neighbourhoods of polydiscs $\CPD_N(0,r)$ and $\CPD_N(0,\rho)$, respectively.
A straightforward calculation allows us to derive the equality
\begin{equation}\label{eq_f_s_g_coordinate_wise}
(f*g)(z)=\left(\frac{1}{2\pi i}\right)^N\int_{\rho^{-1}\TT^N} f(z\zeta) \, g\left(\zeta_1^{-1},\ldots,\zeta_N^{-1}\right)\frac{d\zeta_1}{\zeta_1}\ldots\frac{d\zeta_N}{\zeta_N}
\end{equation}
for $z=(z_1,\ldots,z_N)\in\CPD_N(0,r\rho)$.
Its one-dimensional version was extensively used in research of Hadamard product in one complex variable.
Although in this paper we mostly rely on different tools, the above fact will come in useful at some point in Section \ref{sect_h11_s_D_for_Runge}.

\subsection{Integral formula for the $*$ product with special weights}

Take a bounded smooth domain $0\in D\subset\CC^N$, a neighbourhood $V$ of $\partial D$ and a smooth map $\varphi:V\to\CC^N$ such that $\zeta\bullet\varphi(\zeta)=1$ for all $\zeta\in V$.
If a function $f$ is holomorphic in a neighbourhood of $\overline{D}$ and $\sum_{\alpha\in\NN^N} f_\alpha z^\alpha$ is its Taylor series expansion at the origin, then from the considerations made in \cite[Section 2.1]{zajac} it follows that
$$f_\alpha=c_N\frac{(N+|\alpha|-1)!}{\alpha!}\int_{\partial D} f(\zeta)\varphi(\zeta)^\alpha\omega_\varphi(\zeta),$$
where $c_N$ is a constant dependent only on $N$ and $\omega_\varphi$ is certain smooth $(N,N-1)$ form on $V$ dependent only on $N$ and $\varphi$.
Now, if $g\in\OO_0$ has the Taylor series expansion $\sum_{\alpha\in\NN^N}g_\alpha z^\alpha$, then for $z$ close to $0$ one has that
\begin{equation}\label{eq_sum_N_factorial_f_alpha_g_alpha_z_alpha_integral_formula}
\sum_{\alpha\in\NN^N}\frac{\alpha!(N-1)!}{(|\alpha|+N-1)!} f_\alpha g_\alpha z^\alpha = c_N(N-1)!\int_{\partial D}f(\zeta)g(\varphi(\zeta)z)\omega_\varphi(\zeta).
\end{equation}
We will make use of this equality in Section \ref{sect_D_s_G_for_DN_and_Runge}.

\section{Description of $D*G$ for $D$ being of special class and $G$ being an arbitrary Runge domain}\label{sect_D_s_G_for_DN_and_Runge}

Our goal in this part of the paper is to demonstrate Theorem \ref{th_D_s_G_when_D_from_DN}.
For the Reader's convenience, we begin by recalling definition and elementary facts regarding the $*$-product of domains.

\begin{AS}
Throughout this section we assume that $N\geq 2$ and $D$, $G$ are domains in $\CC^N$ containing the origin. 
\end{AS}

\begin{DEF}
Assume that at least one of $D$, $G$ is Runge domain.
Proposition 4.1 from \cite{zajac} establishes existence of the largest domain $0\in\Omega\subset\CC^N$ having the property that the image of $\OO_{0,D}\times\OO_{0,G}$ by $*$ lies in $\OO_{0,\Omega}$.
We define $D*G$ as this largest $\Omega$.
The referenced fact guarantees that $D*G$ itself is Runge domain.

For every pair of functions $f\in\OO(D)$ and $g\in\OO(G)$ there exists the only element of $\OO(D*G)$ equal to $f*g$ near the origin.
Following \cite{zajac}, we denote it by $f*_{D,G}g$.
We then obtain the bilinear mapping $$*_{D,G}:\OO(D)\times\OO(G)\ni (f,g)\mapsto f*_{D,G}g\in\OO(D*G).$$
From the closed graph theorem it follows that $*_{D,G}$ is separately continuous.
This, together with \cite[page 88, Corollary 1]{schaefer}, guarantees that it is jointly continuous.
\end{DEF}

In a similar fashion we introduce the $*$-product of a germ from $\OO_0$ and a domain containing the origin.

\begin{DEF}
Assume that $G$ is Runge domain.
If $\tau\in\OO_0$, then, in virtue of \cite[Proposition 4.1]{zajac}, there exists the largest domain $0\in\Omega\subset\CC^N$ such that the image of $\OO_{0,G}$ by the map $g\mapsto \tau *g$ lies in $\OO_{0,\Omega}$.
We denote this largest $\Omega$ by $\tau *G$.
As previously, $\tau *G$ is Runge domain.

For each $g\in\OO(G)$ there exists the only function from $\OO(\tau *G)$ equal to $\tau *g$ near the origin.
Denote this function by $\tau *_G g$.
The closed graph theorem yields that the linear operator $$\OO(G)\ni g\mapsto\tau *_G g\in\OO(\tau *G)$$ is continuous.
\end{DEF}

\begin{DEF}
Similarly as in \cite{zajac} we introduce the compact set
$$D^*:=\left\lbrace \xi\in\CC^N:\xi\bullet z\neq 1\text{ for all }z\in D\right\rbrace.$$
If $\xi\in D^*$, then the function 
\begin{equation}\label{eq_h_xi_definition}
h_\xi(z):=\left(1-z\bullet\xi\right)^{-N}.
\end{equation}
belongs to $\OO(D)$ and
\begin{equation}\label{eq_h_xi_taylor_series_0}
h_\xi(z)=\sum_{\alpha\in\NN^N}\frac{(|\alpha|+N-1)!}{\alpha!(N-1)!}\xi^\alpha z^\alpha
\end{equation}
is its Taylor series expansion at the origin.
\end{DEF}

\begin{LEM}\label{lem_upper_bound_on_tau_s_G_if_runge_domain}
Assume that $G$ is Runge domain and $U\subset\CC^N$ is an open set.
If a domain $\Omega\subset\CC^N$ contains the origin and $h_\xi *g\in\OO_{0,\Omega}$ for all $\xi\in U$ and $g\in\OO_{0,G}$, then $$U\cdot\Omega\subset h_{\NONES{N}}*G.$$
\end{LEM}

\begin{proof}
We need to show that $\xi\Omega\subset h_{\NONES{N}}*G$ for each $\xi\in U$.
If $\xi\in U\cap(\CC_*)^N$, then, in view of \eqref{eq_h_xi_taylor_series_0}, for $g\in\OO_{0,G}$ and $z$ close to $0$ we have
$$(h_\xi *g)(z) = (h_{\NONES{N}}*g)(\xi z),$$
so the germ of the function on the right hand side belongs to $\OO_{0,\Omega}$.
Hence, $h_{\NONES{N}}*g\in\OO_{0,\xi\Omega}$.
This holds for every $g\in\OO_{0,G}$, so the definition of $h_{\NONES{N}}*G$ yields that it indeed contains $\xi\Omega$.
On the other hand, if $\xi\in U\setminus(\CC_*)^N$, then we can take a number $r>0$ so that $\xi+r\TT^N\subset U\cap(\CC_*)^N$.
From the previous considerations it follows that $$(\xi+r\TT^N)\cdot\Omega\subset h_{\NONES{N}}*G.$$
Since the set on the right hand side is Runge domain, for each $z\in\Omega$ it contains the polynomial hull of $(\xi+r\TT^N)\cdot z$ and, in particular, the point $\xi\cdot z$ itself.
\end{proof}

\begin{LEM}\label{lem_upper_bound_on_D_s_G}
If $D$ is a pseudoconvex domain and $G$ is Runge domain, then
$$D^*\cdot (D*G)\subset h_{\NONES{N}}*G.$$
Consequently,
$$D*G\subset\CCO\lbrace z\in\CC^N: zD^*\subset h_{\NONES{N}}*G\rbrace.$$
\end{LEM}

\noindent
It is worthy to note that, as $D^*$ is compact, the set under $\CCO$ above is open.

\begin{proof}
Fix an arbitrary domain $0\in\Omega\subset\subset D*G$.
From the continuity of $*_{D,G}$ it follows that we can find a constant $C>0$ and compact sets $K\subset D$ and $L\subset G$ such that $K$ is holomorphically convex in $D$, $0\in\INT K$ and
\begin{equation}\label{eq_luboDsG_f_sDG_g_leq_C_f_K_g_L}
\|f*_{D,G}g\|_{\Omega}\leq C\|f\|_K\|g\|_L 
\end{equation}
for all $f\in\OO(D)$ and $g\in\OO(G)$.

Take a neighbourhood $U$ of $D^*$ such that $\eta\bullet z\neq 1$ for all $z\in K$ and $\eta\in U$.
If $\eta\in U$, then $h_\eta$ is holomorphic in a neighbourhood of $K$, so, by the Oka-Weil theorem, it can be approximated uniformly on $K$ by a sequence $(f_n)_{n\in\NN}\subset\OO(D)$.
Hence, if $g\in\OO(G)$, then \eqref{eq_luboDsG_f_sDG_g_leq_C_f_K_g_L} gives that the functions $f_n *_{D,G}g$ form a Cauchy sequence with respect to the supremum norm on $\Omega$.
This means that they converge in $\OO(\Omega)$ and, thanks to the fact that $0\in\INT K$, the limit has to be equal to $h_\eta *g$ near the origin.
Consequently, $h_\eta *g\in\OO_{0,\Omega}$ for every $g\in\OO_{0,G}$ and $\eta\in U$.
Now, Lemma \ref{lem_upper_bound_on_tau_s_G_if_runge_domain} allows us to conclude that $D^*\cdot\Omega\subset U\cdot\Omega\subset h_{\NONES{N}}*G$, what completes the proof.
\end{proof}

\begin{LEM}\label{lem_cc0_nondecreasing_union}
Let $0\in W$ and $0\in D_0\subset D_1\subset D_2\subset\ldots$ be domains in $\CC^N$ such that $\bigcup_{n\in\NN}D_n=D$.
Set $$\Omega_n=\CCO\lbrace z\in\CC^N: zD_n^*\subset W\rbrace,\quad \Omega=\CCO\lbrace z\in\CC^N: zD^*\subset W\rbrace.$$
Then $\Omega_0\subset\Omega_1\subset\Omega_2\subset\ldots$ and $\bigcup_{n\in\NN}\Omega_n=\Omega$.
\end{LEM}

\begin{proof}
Clearly, $D^*\subset D_{n+1}^*\subset D_n^*$, so $\Omega_n\subset\Omega_{n+1}\subset\Omega$.
To show that $\Omega$ is contained in $\bigcup_{n\in\NN}\Omega_n$, take an arbitrary connected compact set $0\in K\subset\Omega$.
One has that $K\cdot D^*\subset W$, so $K\cdot U\subset W$ for a neighbourhood $U$ of $D^*$.
The sequence $(D_n^*)_{n\in\NN}$ decreases and it is straightforward to check that $D^*=\bigcap_{n\in\NN}D_n^*$.
Therefore, $D_{n_0}^*\subset U$ for some $n_0$, what gives that $K\cdot D_{n_0}^*\subset W$.
Hence, $K\subset\Omega_{n_0}$, because $K$ is connected and contains the origin.
From this we conclude that $\Omega\subset\bigcup_{n\in\NN}\Omega_n$.
\end{proof}

\begin{DEF}
Similarly as in \cite{zajac} we define $\mathcal{D}_N$ as the family of all domains in $\CC^N$ containing the origin which are countable unions of non-decreasing sequences of bounded smooth linearly convex domains. 
\end{DEF}

Recall that $D$ is called \emph{linearly convex} if through every point of $\CC^N\setminus D$ one can pass an affine complex hyperplane disjoint from $D$.
As it is described in \cite[Remark 4.11]{zajac}, each element of $\mathcal{D}_N$, being a union of a non-decreasing sequence of Runge domains, is also Runge domain.

\begin{THE}\label{th_D_s_G_when_D_from_DN}
If $D\in\mathcal{D}_N$ and $G$ is Runge domain, then $$D*G=\CCO\lbrace z\in\CC^N: zD^*\subset h_{\NONES{N}}*G\rbrace.$$
\end{THE}

\begin{proof}
The left-to-right inclusion was established in Lemma \ref{lem_upper_bound_on_D_s_G}, so it remains to prove the opposite one.
In virtue of Lemma \ref{lem_cc0_nondecreasing_union} and \cite[Proposition 4.5]{zajac} it suffices to restrict our considerations to the case when $D$ is bounded, smooth and linearly convex.
Then there exists a smooth map $\nu_D$ from a neighbourhood of $\partial D$ to $\CC^N$ such that at each point $w\in\partial D$ its value $\nu_D(w)$ is the unit outward normal vector for $D$ at $w$.
Hence, for $w\in\partial D$ the equation $(z-w)\bullet\BNU_D(w)=0$ describes the only complex hyperplane passing through $w$ and disjoint from $D$.
In particular, $w\bullet\BNU_D(w)\neq 0$, as $0\in D$.
This means that the mapping $$\varphi:w\mapsto\BNU_D(w)\cdot (w\bullet\BNU_D(w))^{-1}$$ is well-defined and smooth in a neighbourhood $V$ of $\partial D$.
Moreover, we have that $\varphi(\partial D)\subset D^*$ and $w\bullet\varphi(w)=1$ for $w\in V$.

Denote by $\Omega$ the set on the right hand side of the conclusion and take a domain $0\in U\subset\subset\Omega$.
From the definition of $\Omega$ it follows that $U\cdot\varphi(\partial D')\subset h_{\NONES{N}}*G$ for a sufficiently large smooth domain $0\in D'\subset\subset D$ such that $\partial D'\subset V$.
If $f\in\OO(D)$ and $g\in\OO(G)$, then, by \eqref{eq_sum_N_factorial_f_alpha_g_alpha_z_alpha_integral_formula}, for $z$ lying near the origin it holds that
$$(f*g)(z) = c_N(N-1)!\int_{\partial D'}f(\zeta)(h_{\NONES{N}}*_G g)(\varphi(\zeta)z)\omega_{\varphi}(\zeta).$$
The integral on the right hand side defines a function of the variable $z$ which is holomorphic on $U$.
Consequently, $f*g\in\OO_{0,U}$.
Since $f$, $g$ and $U$ were taken arbitrarily, we conclude that $\Omega\subset D*G$.
\end{proof}

\section{Description of $h_{(1,1)}*D$ for Runge domains}\label{sect_h11_s_D_for_Runge}

This part is devoted to demonstration of Theorem \ref{th_description_of_h11_s_D}, which completes, although only in the two-dimensional case, the description of the star product of domains established in Theorem \ref{th_D_s_G_when_D_from_DN}.
Recall that $h_{(1,1)}$ is the function given by the formula \eqref{eq_h_xi_definition}, that is,
$$h_{(1,1)}(z_1,z_2)=(1-z_1-z_2)^{-2}=\sum_{\alpha_1,\alpha_2\in\NN}\frac{(\alpha_1+\alpha_2+1)!}{\alpha_1!\,\alpha_2!} z_1^{\alpha_1}z_2^{\alpha_2}.$$

\begin{AS}
In this section we assume that $D$ is a domain in $\CC^2$ containing the origin.
\end{AS}

\begin{DEF}
To simplify certain statements in this section, let us say that an open set $U\subset\CC_*$ \emph{separates} $0$ and $\infty$ if $U$ contains a loop homotopic in $\CC_*$ to the loop $\zeta\mapsto\zeta$.
This is equivalent to saying that $0$ and $\infty$ lie in different connected components of $\CCS\setminus U$.
\end{DEF}

For a point $z=(z_1,z_2)\in\CC^2$ introduce the mapping $I_z:\CC_*\to\CC^2$ as
$$I_z(\zeta):=(z_1(1+\zeta), z_2(1+\zeta^{-1})).$$
One has that $I_z(-1)=(0,0)$, so $I_z^{-1}(D)$ is non-empty.
It is also important that $I_z$ is an injective proper map when $z\in(\CC_*)^2$.

\begin{THE}\label{th_description_of_h11_s_D}
If $D$ is Runge domain, then
$$h_{(1,1)}*D=\CCO\left\lbrace z\in\CC^2:\text{the set }I_z^{-1}(D)\text{ separates }0\text{ and }\infty\right\rbrace.$$
\end{THE}

\noindent
Directly from the definition of separating is follows that the set under $\CCO$ above is open.
It also contains $(0,0)$, because $I_{(0,0)}^{-1}(D)=\CC_*$.

\begin{RM}
Assume that $D$ is Runge domain and take $z\in\CC^2$.
The open set $I_z^{-1}(D)$ is then $\OO(\CC_*)$-convex in the sense that $\HULL{L}{\CC_*}\subset I_z^{-1}(D)$ for every compact set $L\subset I_z^{-1}(D)$.
This means that every connected component of $\CCS\setminus I_z^{-1}(D)$ contains $0$ or $\infty$ (possibly both of them).
Consequently, the latter set is connected if and only if $I_z^{-1}(D)$ does not separate $0$ and $\infty$.
\end{RM}

\begin{OBS}\label{obs_Izm1_K_polynomially_convex_when_Izm1_D_not_separates}
If an open set $\Omega\subset\CC^2$ and a point $z\in(\CC_*)^2$ are such that $I_z^{-1}(\Omega)$ does not separate $0$ and $\infty$, then for every compact polynomially convex set $K\subset\Omega$ the pre-image $I_z^{-1}(K)$ is either empty or polynomially convex.
\end{OBS}

\begin{proof}
First, note that the set $L:=I_z^{-1}(K)$, if non-empty, has to be compact and holomorphically convex in $\CC_*$, what means that every connected component of $\CCS\setminus L$ contains $0$ or $\infty$.
But since $I_z^{-1}(\Omega)$ does not separate $0$ and $\infty$, these points lie in the same connected component of $\CCS\setminus L$.
Hence, the set $\CCS\setminus L$ is connected, so $L$ is polynomially convex.
\end{proof}

\begin{LEM}\label{lem_K_union_I_z_rho_T_polynomially_convex}
Let $K\subset\CC^2$ be a compact polynomially convex set, $z\in(\CC_*)^2$ and $\varrho\in (0,\infty)$.
If $I_z^{-1}(K)$ is empty or polynomially convex and $$I_z^{-1}(K)\cap\varrho\CDD=\varnothing,$$ then the union $K\cup I_z(\varrho\TT)$ is polynomially convex.
\end{LEM}

\begin{proof}
Write $z=(z_1,z_2)$ and define
$$\mu(w_1,w_2):=(w_1-z_1)(w_2-z_2)-z_1z_2,\quad(w_1,w_2)\in\CC^2.$$
Clearly, $M:=\mu^{-1}(0)=I_z(\CC_*)$ is a complex submanifold of $\CC^2$ and the map $I_z:\CC_*\to M$ is a biholomorphism.
By the assumptions, the union $I_z^{-1}(K)\cup\varrho\TT$ is holomorphically convex in $\CC_*$.
This implies that the set $(K\cap M)\cup I_z(\varrho\TT)$, being its image by $I_z$, is holomorphically convex in $M$ and thus polynomially convex in $\CC^2$ (use e.g. \cite[Theorem 7.4.8]{hormander_introduction}).
The conclusion now follows directly from the subsequent general lemma.
\end{proof}

\begin{LEM}
Let $V$ be an analytic subset of $\CC^N$ and let $K\subset\CC^N$, $L\subset V$ be compact sets.
If both $K$ and $(K\cap V)\cup L$ are polynomially convex, then so is $K\cup L$.
\end{LEM}

\noindent
Note that the sets $K$ and $L$ do not have to be disjoint.

\begin{proof}
It is known (see e.g. \cite[Theorems 6.5.2 and 7.1.5]{hormander_introduction}) that the sheaf of germs of holomorphic functions vanishing on $V$ is a coherent analytic sheaf on $\CC^N$.
Therefore, if $X\subset\CC^N$ is a compact, polynomially convex set (intersecting $V$ or not), $U$ is a neighbourhood of $X$, $F\in\OO(U)$ and $F|_{U\cap V}\equiv 0$, then $F$ is a section of this sheaf and it can be approximated uniformly on $X$ by global sections, that is, by elements of $\OO(\CC^N)$ vanishing on $V$.
This essential fact is a consequence of \cite[Theorem 7.2.7]{hormander_introduction}.

Fix a point $z_0\in\CC^N\setminus (K\cup L)$.
We are going to show that $z_0\not\in\HULLP{K\cup L}$.
If $z_0\not\in V$, then \cite[Theorem 7.2.11]{hormander_introduction} provides $f\in\OO(\CC^N)$ vanishing on $V$ and such that $f(z_0)=1$.
On the other hand, there exists $g\in\OO(\CC^N)$ having $g(z_0)=1$ and $\|g\|_K<1$.
Hence, for sufficiently large number $n$ the function $g^n f$ maps $z_0$ to $1$ and $K\cup L$ into $\DD$.

It remains to consider the case when $z_0\in V$.
Since $(K\cap V)\cup L$ is polynomially convex, one can find another compact polynomially convex set $A\subset\CC^N$ such that $$(K\cap V)\cup L\subset\INT A\text{ and }z_0\not\in A.$$
Take a sequence $(p_n)_{n\in\NN}\subset\OO(\CC^N)$ such that $$p_n(z_0)=1\text{ and }\|p_n\|_A\to 0\text{ when }n\to\infty.$$
The set $(\CC^N\setminus V)\cup\INT A$ is an open neighbourhood of $K$, so it has a pseudoconvex open subset $\Omega$ containing $K$.
This means that $\Omega\cap V\subset\INT A$ and hence $p_n\to 0$ on $\Omega\cap V$.
In virtue of \cite[Theorem 13.1]{bungart}, there exists a continuous linear extension operator from the Banach space of bounded holomorphic functions on $\Omega\cap V$ into $\OO(\Omega)$.
Applying it to $p_n$'s we obtain a sequence $(g_n)_{n\in\NN}\subset\OO(\Omega)$ convergent to $0$ in $\OO(\Omega)$ and such that $g_n-p_n=0$ on $\Omega\cap V$.
As it was described in the first paragraph of the proof, for every $n$ one can find a function $q_n\in\OO(\CC^N)$ so that $q_n|_V\equiv 0$ and $\|q_n+p_n-g_n\|_K<\frac 1n$.
Finally, set $f_n:=p_n+q_n$.
Every $f_n$ is an entire function and $$f_n=p_n\text{ on }V\text{ and }\|f_n-g_n\|_K<\frac 1n.$$
This implies that $f_n(z_0)=1$ and $f_n\to 0$ on $K\cup L$ uniformly when $n\to\infty$, so for large $n$ it holds that $|f_n(z_0)|>\|f_n\|_{K\cup L}$, as desired.
\end{proof}

For a holomorphic function $f$ of two variables define 
$$\Lambda(f)(z_1,z_2):=f(z_1,z_2)+z_1\,\frac{\partial f}{\partial z_1}(z_1,z_2).$$
For every domain $\Omega\subset\CC^2$ the mapping $f\mapsto\Lambda(f)$ defines a continuous linear operator on $\OO(\Omega)$.

\begin{LEM}\label{lem_h11_s_f_integral_Lambda_f_Iz}
If $\gamma$ is a loop in $\CC_*$ homotopic to the loop $\zeta\mapsto\zeta$ and $f$ is a polynomial in $\CC^2$, then
$$(h_{(1,1)}*_{\CC^2}f)(z)=\frac{1}{2\pi i}\int_{\gamma}(1+\zeta^{-1})\,\Lambda(f)(I_z(\zeta))d\zeta$$
for all $z\in\CC^2$.
\end{LEM}

\begin{proof}
Fix a number $\rho\in(0,\frac 12)$ and a polynomial $f$.
From \eqref{eq_f_s_g_coordinate_wise} it follows that
$$(h_{(1,1)}*_{\CC^2} f)(z)=\left(\frac{1}{2\pi i}\right)^2\int_{\rho^{-1}\TT^2}\frac{f(z_1\zeta_1,z_2\zeta_2)\zeta_1\zeta_2}{(\zeta_2-1)^2(\zeta_1-\zeta_2(\zeta_2-1)^{-1})^2} d\zeta_1 d\zeta_2,$$
when $z=(z_1,z_2)\in\CC^2$.
If $\zeta_2\in\rho^{-1}\TT$, then the point $\zeta_2(\zeta_2-1)^{-1}$ lies in $\rho^{-1}\DD$, so, in view of the Cauchy formula, 
\begin{multline*}
\frac{1}{2\pi i}\int_{\rho^{-1}\TT}\frac{f(z_1\zeta_1,z_2\zeta_2)\zeta_1}{(\zeta_1-\zeta_2(\zeta_2-1)^{-1})^2} d\zeta_1\\
=\left.\frac{d}{d\zeta_1} \Bigl(f(z_1\zeta_1,z_2\zeta_2)\zeta_1\Bigr) \right|_{\zeta_1=\zeta_2(\zeta_2-1)^{-1}}
=\Lambda(f)\left(\frac{z_1\zeta_2}{\zeta_2-1},z_2\zeta_2\right).
\end{multline*}
Therefore,
$$(h_{(1,1)}*_{\CC^2} f)(z)=\frac{1}{2\pi i}\int_{\rho^{-1}\TT}\frac{\zeta_2}{(\zeta_2-1)^2}\;\Lambda(f)\left(\frac{z_1\zeta_2}{\zeta_2-1},z_2\zeta_2\right)d\zeta_2.$$
A homotopy argument allows us to integrate over $\rho^{-1}\TT+1$ instead of $\rho^{-1}\TT$.
Then, after changing the variable via $\zeta:=(\zeta_2-1)^{-1}$, we obtain the equality from the conclusion with the integral over $\rho\TT$.
It is not affected by replacing $\rho\TT$ by $\gamma$, because the integrated function of $\zeta$ is holomorphic on $\CC_*$.
\end{proof}

\begin{proof}[Proof of Theorem \ref{th_description_of_h11_s_D}]
Denote by $\Omega$ the set on the right hand side of the conclusion, that is,
$$\Omega:=\CCO\left\lbrace z\in\CC^2:\text{the set }I_z^{-1}(D)\text{ separates }0\text{ and }\infty\right\rbrace.$$
The proof of the equality $h_{(1,1)}*D=\Omega$ is divided into a few steps.

\smallskip
\textsc{Step 1:} We show the inclusion $\Omega\subset h_{(1,1)}*D$.

\smallskip
Fix a function $f\in\OO(D)$ and take a sequence $(f_n)_{n\in\NN}$ of polynomials convergent to $f$ locally uniformly on $D$.
Then the functions $h_{(1,1)}*_{\CC^2} f_n$ tend to $h_{(1,1)}*_D f$ in the same manner on $h_{(1,1)}*D$.
We claim that they form a sequence convergent on $\Omega$ as well.
Fix a point $a\in\Omega$ and take a loop $\gamma:\TT\to I_a^{-1}(D)$ homotopic in $\CC_*$ to the identity loop.
Choose a closed ball $B\subset\Omega$ centered at $a$ so that $I_z(\gamma(\TT))\subset D$ for all $z\in B$.
From Lemma \ref{lem_h11_s_f_integral_Lambda_f_Iz} it follows that
$$(h_{(1,1)}*_{\CC^2} f_n)(z)=\frac{1}{2\pi i}\int_{\gamma}(1+\zeta^{-1})\,\Lambda(f_n)(I_z(\zeta))d\zeta$$
for all $z\in\CC^2$, $n\in\NN$.
Since $\Lambda(f_n)\to\Lambda(f)$ in $\OO(D)$ as $n\to\infty$, the integrals on the right hand side of the above equality converge uniformly with respect to $z\in B$ to the identical integral with $f_n$ repalced by $f$.
Consequently, the sequence of polynomials $h_{(1,1)}*_{\CC^2}f_n$ is uniformly convergent on $B$.
Hence, it does converge locally uniformly on $\Omega$, because $a\in\Omega$ was chosen arbitrarily.
If $g\in\OO(\Omega)$ is the limit, then clearly $g=h_{(1,1)}*f$ near the origin.
This means that $h_{(1,1)}*f\in\OO_{0,\Omega}$ for every $f\in\OO(D)$, so $\Omega\subset h_{(1,1)}*D$.

\smallskip
\textsc{Step 2:} We prove the inclusion $(h_{(1,1)}*D)\cap(\CC_*)^2\subset\Omega$.

\smallskip
Suppose, to the contrary, that it is not valid, and choose a point $z\in (h_{(1,1)}*D)\cap (\CC_*)^2\cap\partial\Omega$.
There exist a constant $C>0$ and a polynomially convex compact set $K\subset D$ such that
\begin{equation}\label{eq_tdh11sD_h11_sD_f_U_leq_C_f_K_1}
|(h_{(1,1)}*_D f)(z)|\leq C\|f\|_K,\quad f\in\OO(D).
\end{equation}
Take a number $\varrho>0$ so that $I_z^{-1}(K)\cap\varrho\CDD=\varnothing$.
Since $z\in\partial\Omega$, the set $I_z^{-1}(D)$ does not separate $0$ and $\infty$, so from Observation \ref{obs_Izm1_K_polynomially_convex_when_Izm1_D_not_separates} and Lemma \ref{lem_K_union_I_z_rho_T_polynomially_convex} it follows that the union $K\cup I_z(\varrho\TT)$ is polynomially convex.
Hence, there exists a compact set $L\subset\CC^2$ such that $I_z(\varrho\TT)\subset\INT L$, $K\cap L=\varnothing$ and $K\cup L$ is polynomially convex (one can justify it making use, for example, of the Kallin Lemma \cite[Theorem 1.6.19]{stout}).
This allows us to employ the Oka-Weil theorem to get a sequence $(f_n)_{n\in\NN}$ of polynomials in $\CC^2$ uniformly convergent to $0$ on $K$ and to $1$ on $L$.
Lemma \ref{lem_h11_s_f_integral_Lambda_f_Iz} implies that
$$(h_{(1,1)}*_D f_n)(z)=\frac{1}{2\pi i}\int_{\varrho\TT}(1+\zeta^{-1})\,\Lambda(f_n)(I_z(\zeta))d\zeta.$$
If $n\to\infty$, then the left hand side goes to $0$, in view of \eqref{eq_tdh11sD_h11_sD_f_U_leq_C_f_K_1}.
On the other hand, the integrals converge to $1$, because $\Lambda(f_n)\to 1$ on $I_z(\varrho\TT)$.
A contradiction.

\smallskip
\textsc{Step 3:} We show that if $(z_1,0)\in h_{(1,1)}*D$, then $(z_1,0)\in D$.
Thanks to evident symmetry, it will also mean that $(0,z_2)\in D$ when $(0,z_2)\in h_{(1,1)}*D$.

\smallskip
Suppose, to the contrary, that $(z_1,0)\not\in D$.
One can find a constant $C>0$ and a compact polynomially convex set $K\subset D$ satisfying
\begin{equation}\label{eq_tdh11sD_h11_sD_f_U_leq_C_f_K_2}
|(h_{(1,1)}*_D f)(z_1,0)|\leq C\|f\|_K,\quad f\in\OO(D).
\end{equation}
Take a closed ball $B$ centered at $(z_1,0)$ so small that $K\cap B=\varnothing$ and $K\cup B$ is polynomially convex.
In view of the Oka-Weil theorem, there exists a sequence $(f_n)_{n\in\NN}$ of polynomials in $\CC^2$ uniformly convergent to $0$ on $K$ and to $1$ on $B$.

From Lemma \ref{lem_h11_s_f_integral_Lambda_f_Iz} is follows that
\begin{equation}\label{eq_tdh11sD_h11_h11_s_fn_equals_fn_zeta_dfn_dz1}
(h_{(1,1)}*_{\CC^2} f_n)(z_1,0) = \Lambda(f_n)(z_1,0) = f_n(z_1,0)+z_1\,\frac{\partial f_n}{\partial z_1}(z_1,0).
\end{equation}
Now, if $n\to\infty$, then \eqref{eq_tdh11sD_h11_sD_f_U_leq_C_f_K_2} implies that the left hand side of \eqref{eq_tdh11sD_h11_h11_s_fn_equals_fn_zeta_dfn_dz1} goes to $0$, while the right hand side converges to $1$.
This contradiction completes the proof of this step.

\smallskip
\textsc{Step 4:} We obtain the inclusion $h_{(1,1)}*D\subset\Omega$.

\smallskip
It remains to prove that if $(z_1,0)\in h_{(1,1)}*D$, then $(z_1,0)\in\Omega$ (the same statement for $(0,z_2)$ can be demonstrated identically).
Take such a point $(z_1,0)$ and choose a curve $\gamma:[0,1]\to h_{(1,1)}*D$ so that $\gamma(0)=(0,0)$, $\gamma(1)=(z_1,0)$ and $\gamma(t)\in(\CC_*)^2$ for $t\in(0,1)$.
The conclusion of Step 2 guarantees that $\gamma(t)\in\Omega$ every $t\in(0,1)$, so $(z_1,0)\in\overline{\Omega}$.
Moreover, from Step 3 we know that $(z_1,0)\in D$, so $\epsilon\DD_*\subset I_{(z_1,0)}^{-1}(D)$ for a small number $\epsilon>0$.
Hence, the latter set separates $0$ and $\infty$, what gives that $(z_1,0)\in\Omega$.
The proof is complete.
\end{proof}


\end{document}